\documentclass[12pt]{article}
\usepackage{amsmath,amscd}
\usepackage{amssymb,latexsym,amsthm}
\usepackage{color}

\usepackage{amssymb}
\usepackage{color}
\usepackage{amsmath,amsthm,amscd}
\usepackage[latin1]{inputenc}
\usepackage{lscape}
\usepackage{fancyhdr}
\usepackage{amsfonts}
\usepackage{pb-diagram}
\usepackage{hyperref}

\numberwithin{equation}{section}
\newtheorem{theorem}{Theorem}[section]

\newtheorem{proposition}[theorem]{Proposition}

\newtheorem{corollary}[theorem]{Corollary}

\theoremstyle{definition}

\newtheorem{definition}[theorem]{Definition}

\newtheorem{remark}[theorem]{Remark}

\title{\textbf{Homological properties of some quantum Heisenberg algebras}}
\author{Samuel A. Lopes
\footnote{CMUP, Departamento de Matem\'atica, Faculdade de Ci\^encias, Universidade do Porto, Rua do Campo Alegre s/n, 4169--007 Porto, Portugal. https://orcid.org/0000-0002-1292-0977.}\\H\'ector Su\'arez  \footnote{Grupo de \'Algebra y An\'alisis, Universidad Pedag\'ogica y Tecnol\'ogica
de Colombia, Tunja, Colombia. https://orcid.org/0000-0003-4618-0599.}\\ Y\'esica Su\'arez
\footnote{Grupo de \'Algebra y An\'alisis, Universidad Pedag\'ogica
y Tecnol\'ogica de Colombia, Tunja, Colombia. \\https://orcid.org/0000-0002-7373-4852
.}}
\date{}
\begin{document}
\maketitle
\begin{abstract}
This paper presents a survey of several classes of Heisenberg-type algebras, including classical, quantum, and generalized variants, as well as their deformations and extensions. We explore the structural and homological properties of these algebras, focusing on their realization as (graded) skew PBW extensions and (graded) iterated Ore extensions. Special attention is given to the Koszul, Artin--Schelter regular, and (skew) Calabi--Yau properties of these algebras. We also examine the relationships between different families, such as quantum Heisenberg algebras, generalized Heisenberg algebras, and quantum generalized Heisenberg algebras, highlighting conditions under which these algebras exhibit desirable homological behaviors. Our results unify and extend several known constructions, offering a comprehensive perspective on the homological landscape of Heisenberg-type algebras.

\bigskip


\bigskip

\end{abstract}

\section{Introduction}

Heisenberg-type algebras, originating from the (universal enveloping algebra of the) three-dimensional Heisenberg Lie algebra, have inspired a rich variety of generalizations and deformations. These include quantum Heisenberg algebras (\cite{Kirkman1993}), two-parameter deformations (\cite{Hayashi1990,Gaddis2016}), generalized Heisenberg algebras (\cite{Curado2001, Lu2015}), and more recently, quantum generalized Heisenberg algebras (\cite{Lopes2022Q, Lopes2022}). These algebras not only generalize classical structures but also serve as fertile ground for exploring interesting homological properties in noncommutative algebra.

In this paper, we present a survey of several such algebras, focusing on their algebraic and homological features. We examine how these algebras can be realized as (graded) skew PBW extensions and (graded) iterated Ore extensions, which provides a unifying framework for understanding their structure. We investigate conditions under which these algebras are Koszul, Artin--Schelter regular, and (skew) Calabi--Yau, and we explore how these properties are preserved under various constructions, such as Ore extensions and PBW deformations.

Our approach highlights the interplay between structural ring-theoretic characteristics and homological properties, offering new insights into the behavior of Heisenberg-type algebras. In particular, we provide an overview of skew PBW extensions  and relate this construction to our main examples, including quantum generalized Heisenberg algebras. Along the way, we clarify the relationships between different families of Heisenberg algebras and provide new results on their homological properties.

\section{Preliminaries on Artin--Schelter regular and Calabi--Yau algebras, and skew PBW extensions}

Throughout this paper, $K$ is an arbitrary field, $K^*=K\setminus \{0\}$, every ring (algebra) is associative with identity and algebras are $K$-algebras,  $K\langle x_1, \dots , x_{n}\rangle$ is the free associative algebra with $n$ generators, unless otherwise
stated. We will denote by $\mathbb{Z}$, $\mathbb{N}$, $\mathbb{Z}^+$ the sets of all integers, nonnegative integers, and positive integers, respectively.

A graded algebra $B=\oplus_{p\geq 0}B_p$ is called \emph{connected}
if $B_0=K$. In this case, let $K = B/B_{\geq 1}$ be the trivial module. Then $B$ is further said to
be \emph{Artin--Schelter regular} of dimension $d$ if $B$ has finite
global dimension $d$ and
\[Ext^i_{B} (K, B) \cong \left\{
                                         \begin{array}{ll}
                                           K(l), & i= d; \\
                                           0, & i\neq d,
                                           \end{array}
                                       \right.\]
for some $l\in \mathbb{Z}$. Commutative Artin--Schelter regular algebras are polynomial rings. Artin--Schelter regular algebras of global dimension two
and global dimension three were classified in \cite{ArtinSchelter}, but there
are still many open questions about these algebras and the
classification of Artin--Schelter regular algebras of global
dimension greater than three is still open. Different authors have focused on
studying Artin--Schelter regular algebras, especially those of global
dimension four and five (see e.g.\
\cite{Wang,Zhang,Zhang2,Zhun}). For example, in \cite{Zhang2}, Zhang and Zhang constructed $26$ families of Artin--Schelter regular algebras of global dimension four using a special class of
algebras, called double Ore extensions, which they had introduced in \cite{Zhang}. They proved that a connected graded double Ore extension of an Artin--Schelter regular algebra is again Artin--Schelter regular.

The \emph{enveloping algebra} of an algebra  $B$ is the tensor
product $B^e =B\otimes B^{op}$, where $B^{op}$ is the opposite
algebra of $B$. Calabi--Yau algebras were defined
by Ginzburg in \cite{Ginzburg1} and the skew Calabi--Yau algebras
were defined as a generalization of those. A graded algebra $B$ is called \emph{graded skew Calabi--Yau} of dimension $d$ if, as a $B^e$-module, $B$ has a projective resolution that has finite length and such that each term in the projective resolution is finitely generated, and there exists an algebra automorphism  $\nu$
of $B$ such that \[Ext^i_{B^e} (B,B^e) \cong \left\{
                                         \begin{array}{ll}
                                           0, & i\neq d; \\
                                           B^{\nu}(l), & i= d;
                                         \end{array}
                                       \right.\]
 as $B^e$-modules, for some integer $l$.  If $\nu$ is the identity, then $B$ is said to be \emph{Calabi--Yau}. Ungraded Calabi--Yau algebras are defined similarly but
without the degree shift. The automorphism $\nu$ is called the \emph{Nakayama} automorphism of $B$, and is unique up to inner automorphisms of $B$. Thus a skew Calabi--Yau algebra is Calabi--Yau  if and only if its Nakayama automorphism is inner. Reyes, Rogalski and
 Zhang in \cite{ReyesRogalskiZhang} proved that, for connected algebras, the skew
Calabi--Yau property is equivalent  to the Artin--Schelter regular
property. An example of a skew
Calabi--Yau algebra which is not Calabi--Yau is the Jordan plane.

Skew PBW extensions, defined below, are very useful constructions as they often preserve most of the properties referenced above and a few others. They were introduced in \cite{GallegoLezama} and have been profusely studied in the recent years (see for example \cite{ReyesSuarez2019-2,ReyesSuarez2019Radicals,SuarezChaconReyes2021,SuarezReyes2023,SuarezHigueraReyes2024}). 

While many of the algebras considered herein can be realized as Ore extensions, the framework of skew PBW extensions provides a different setting that encompasses some algebras which are not iterated Ore extensions; moreover, it is especially adequate for the study of homological properties of PBW deformations, such as Koszul and (skew) Calabi--Yau structures (see also the comments above Proposition~\ref{prop.twoHisgrskpbw}).

\begin{definition}[\cite{GallegoLezama}, Definition 1]\label{def.skewpbwextensions}
Let $R$ and $A$ be rings. We say that $A$ is a \textit{skew PBW
extension over} $R$ (the ring of coefficients), denoted
$A=\sigma(R)\langle x_1,\dots,x_n\rangle$, if the following
conditions hold:
\begin{enumerate}
\item[\rm (i)]$R$ is a subring of $A$ sharing the same identity element.
\item[\rm (ii)] there exist finitely many elements $x_1,\dots ,x_n\in A$ such that $A$ is a left free $R$-module, with basis the
set of standard monomials
\begin{center}
${\rm Mon}(A):= \{x^{\alpha}:=x_1^{\alpha_1}\cdots
x_n^{\alpha_n}\mid \alpha=(\alpha_1,\dots ,\alpha_n)\in
\mathbb{N}^n\}$.
\end{center}
Moreover, $x^0_1\cdots x^0_n := 1 \in {\rm Mon}(A)$.
\item[\rm (iii)]For each $1\leq i\leq n$ and any $r\in R\ \backslash\ \{0\}$, there exists an
    element $c_{i,r}\in R\ \backslash\ \{0\}$ such that
\begin{equation}\label{eq.coef}
x_ir-c_{i,r}x_i\in R.
\end{equation}
\item[\rm (iv)]For $1\leq i,j\leq n$, there exists $d_{i,j}\in R\ \backslash\ \{0\}$ such that
\begin{equation}\label{eq.variab}
x_jx_i-d_{i,j}x_ix_j\in R+Rx_1+\cdots +Rx_n,
\end{equation}
i.e., there exist elements $r_0^{(i,j)}, r_1^{(i,j)}, \dotsc, r_n^{(i,j)} \in R$ with $x_jx_i - d_{i,j}x_ix_j = r_0^{(i,j)} + \sum_{k=1}^{n} r_k^{(i,j)}x_k$.
\end{enumerate}
\end{definition}

For $X = x^{\alpha} = x_1^{\alpha_1}\cdots x_n^{\alpha_n} \in \text{Mon}(A)$, set $\text{deg}(X) = |\alpha| := \alpha_1 + \cdots + \alpha_n$. From the definition, it follows that every non-zero element $f \in A$ can be uniquely expressed as $f = a_0 + a_1X_1 + \cdots + a_mX_m$, with $a_i \in R$ and $X_i \in \text{Mon}(A)$, for $0 \leq i \leq m$ \cite[Remark 2]{GallegoLezama}.

\begin{proposition}[\cite{GallegoLezama}, Proposition 3] \label{prop.endoderiv}
If $A=\sigma(R)\langle x_1,\dots,x_n\rangle$ is a skew PBW extension, then there exist an injective endomorphism $\sigma_i:R\rightarrow R$ and a $\sigma_i$-derivation $\delta_i:R\rightarrow R$ such that $x_ir=\sigma_i(r)x_i+\delta_i(r)$, for each $1\leq i\leq n$, where $r\in R$.
\end{proposition}

Let $A=\sigma(R)\langle x_1,\dots,x_n\rangle$ be a skew PBW extension over a ring $R$ and $\sigma_i$, $\delta_i$ as in Proposition \ref{prop.endoderiv}. $A$ is said to be {\it bijective} if  $\sigma_i$ is bijective for each $1 \leq i \leq n$ and $d_{i,j}$ is invertible for any $1 \leq i <j \leq n$.  If $\delta_i$ is zero for every $i$, then $A$ is called a skew PBW extension of \textit{endomorphism type}.

\begin{proposition}[\cite{Suarez2017}, Proposition 2.7]\label{prop.grad A}
Let $R=\oplus_{m\geq 0}R_m$ be an $\mathbb{N}$-graded algebra and let $A=\sigma(R)\langle
x_1,\dots, x_n\rangle$ be a bijective skew PBW extension of $R$ satisfying the fo\-llo\-wing two
conditions:
\begin{enumerate}
\item[\rm (i)] $\sigma_i$ is a graded ring homomorphism and $\delta_i : R(-1) \to R$ is a graded $\sigma_i$-derivation for all $1\leq i  \leq n$,
where $\sigma_i$ and $\delta_i$ are as in Proposition \ref{prop.endoderiv} and $R(-1)$ denotes a graded shift.
\item[\rm (ii)]  $x_jx_i-d_{i,j}x_ix_j\in R_2+R_1x_1 +\cdots + R_1x_n$, as in (\ref{eq.variab}) and $d_{i,j}\in R_0$.
\end{enumerate}
For $p\geq 0$, let $A_p$ be the $K$-space generated by the set
\[\Bigl\{r_tx^{\alpha} \mid t+|\alpha|= p,\  r_t\in R_t \text{  and } x^{\alpha}\in {\rm
Mon}(A)\Bigr\}.
\]
Then $A$ is an $\mathbb{N}$-graded algebra with graduation
\begin{equation}\label{eq.grad alg skew}
A=\oplus_{p\geq 0} A_p.
\end{equation}
\end{proposition}

In \cite{Suarez2017}, the second-named author defined the graded skew PBW extensions for an algebra $R$ as a generalization of graded iterated Ore extensions. 

\begin{definition}[\cite{Suarez2017}, Definition 2.6]\label{def. graded skew PBW ext} Let  $A=\sigma(R)\langle x_1,\dots, x_n\rangle$ be a bijective
skew PBW extension of an $\mathbb{N}$-graded algebra
$R=\oplus_{m\geq 0}R_m$. We say that $A$ is a \emph{graded  skew
PBW extension} if $A$ satisfies the conditions (i) and (ii) in
Proposition \ref{prop.grad A}.
\end{definition}

Some properties of graded skew PBW
extensions have been studied in \cite{Suarez2017Calabi,Suarez2017,SuarezReyesYesica2023}.
G\'omez and Su\'arez in \cite{GomezSuarez2020} gave necessary and sufficient conditions for a graded (trimmed) double Ore extension to be a
graded (quasi-commutative) skew PBW extension and, using this fact, they
proved that a graded skew PBW extension $A = \sigma(R)\langle x_1,x_2
\rangle$ of an Artin--Schelter regular algebra $R$ is Artin--Schelter
regular and that every such extension $A$ of a connected skew Calabi--Yau
algebra $R$ of dimension $d$ is skew Calabi--Yau of dimension $d+2$.

A graded connected algebra $B$ is called  \emph{Koszul} if there exists a graded projective resolution of the trivial module $K$
$$ \cdots \to P_{i}\to P_{i-1}\to \cdots \to P_0\to K\to 0$$
such that, for any $i\geq 0$, $P_{i}$ is generated in degree $i$. In \cite[Theorem 5.5]{Suarez2017}, the second-named author proved that if $A$ is a graded skew PBW extension of a finitely presented Koszul algebra $R$, then $A$ is Koszul.

\section{Classical and quantum Heisenberg algebras}\label{sect.qHeis}

The \emph{enveloping algebra of the 3-dimensional Heisenberg Lie algebra} (or \emph{Heisenberg algebra}, for short) is the unital associative algebra $\mathcal{H}$ generated by the variables $x, y, t$, with relations
\begin{equation}\label{eq.Heis}
  tx=xt,\quad yt=ty,\quad yx-xy = t.
\end{equation}
Using (\ref{eq.Heis}), we see that $\mathcal{H}$ is a
skew PBW extension of $K[t]$, i.e., $\mathcal{H} =
\sigma(K[t])\langle x,y\rangle$. Also, $\mathcal{H}$ is the
Ore extension of derivation type $K[t,x][y;\delta]$, where
$\delta(t)=0$ and $\delta(x)=t$.

Note that the Heisenberg algebra is Calabi--Yau of dimension 3 \cite[Proposition 4.6]{He2010} and Artin--Schelter regular \cite[Page 172]{ArtinSchelter}.

Kirkman and Small in \cite{Kirkman1993} studied a $q$-analogue of
$\mathcal{H}$, known as the quantum Heisenberg algebra. For $q\in K^*$
the \emph{quantum Heisenberg algebra} $\mathcal{H}_q$ is the algebra
generated by $t,x,y$ subject to the relations
\begin{equation}\label{eq.quanHeis}
  yx-qxy=t, \quad xt = qtx, \quad yt= q^{-1}ty.
\end{equation}

Note that $\mathcal{H}_q$ is a skew PBW extension of $K[t]$, i.e.,
$\mathcal{H}_q=\sigma(K[t])\langle
x,y\rangle$. Also $\mathcal{H}_q$ is the iterated Ore extension $$K[t][x;
\sigma_x][y; \sigma_y, \delta_y],$$ where $\sigma_x(t)=qt$, $\sigma_y(t)=q^{-1}t$, $\sigma_y(x)=qx$, $\delta_y(t)=0$, $\delta_y(x)=t$.

Gaddis in \cite{Gaddis2016} introduced and studied a two-parameter analog of the Heisenberg
algebra $\mathcal{H}$. For $p,q\in K^*$, the Gaddis \emph{two-parameter quantum Heisenberg algebra} $G\mathcal{H}_{p,q}$ is the
algebra generated by $t,x,y$ subject to the relations
\begin{equation}\label{eq.quantwoHeis}
yx-qxy=t, \quad xt = ptx, \quad yt= p^{-1}ty.
\end{equation}

Note that if $p=q$, the Gaddis two-parameter quantum Heisenberg algebra $G\mathcal{H}_{p,q}$ becomes the quantum Heisenberg algebra: $\mathcal{H}_{q}=G\mathcal{H}_{q,q}$.

According to  \cite{Gaddis2016}, Hayashi \cite{Hayashi1990} had already defined certain quantum algebras in two parameters, in which two of the relations of $G\mathcal{H}_{p,q}$ coincide, but the other has a small modification. The Hayashi \emph{two-parameter quantum Heisenberg algebra} $H\mathcal{H}_{p,q}$ is the
algebra generated by $t,x,y$ subject to the relations
\begin{equation}\label{eq.HquantwoHeis}
yx-qxy=t^2, \quad xt = ptx, \quad yt= p^{-1}ty.
\end{equation}

\begin{proposition}\label{prop.twoqHisSkew}
$G\mathcal{H}_{p,q}$ and $H\mathcal{H}_{p,q}$ are bijective skew PBW extension of endomorphism type  of $K[t]$ and  iterated Ore extensions.
\end{proposition}
\begin{proof}
The second and third relations in  (\ref{eq.quantwoHeis})
correspond to condition (\ref{eq.coef}) of Definition
\ref{def.skewpbwextensions}. The first defining relation in
(\ref{eq.quantwoHeis}) corresponds to condition (\ref{eq.variab}).
The endomorphisms and derivations of Proposition
\ref{prop.endoderiv} are given by $\sigma_x(t)=pt$, $\delta_x(t)=0$
and $\sigma_y(t)=p^{-1}t$, $\delta_y(t)=0$. Therefore
$G\mathcal{H}_{p,q}=\sigma(K[t])\langle x,y\rangle$ is a bijective
skew PBW extension of endomorphism type. Also (see~\cite[Proposition 3.4]{Gaddis2016}) $G\mathcal{H}_{p,q}$ is
the iterated Ore extension $$K[t][x; \sigma_x][y; \sigma_y,
\delta_y],$$ where $\sigma_x(t)=pt$, $\sigma_y(t)=p^{-1}t$,
$\sigma_y(x)=qx$, $\delta_y(t)=0$ and $\delta_y(x)=t$.

With a reasoning analogous to the previous one but changing $\delta_y(x)=t$ to $\delta_y(x)=t^2$ we have that $H\mathcal{H}_{p,q}$ is a bijective skew PBW extension of endomorphism type  of $K[t]$ and an iterated Ore extension.
\end{proof}

Let $A=R[x_1; \sigma_1,\delta_1]\cdots [x_{n}; \sigma_{n},\delta_{n}]$ be an iterated Ore extension. Then $A$ is called a \emph{graded iterated Ore extension} if $x_1,\dots, x_n$ have degree 1 in $A$, $R$ is in degree $0$,
\begin{equation}\label{eq.autiteratOre}
\sigma_i : R[x_1; \sigma_1,\delta_1]\cdots [x_{i-1}; \sigma_{i-1},\delta_{i-1}]  \to R[x_1; \sigma_1,\delta_1]\cdots [x_{i-1}; \sigma_{i-1},\delta_{i-1}]
\end{equation}
is a graded algebra automorphism and
\begin{equation}\label{eq.deriteratOre}
\delta_i : R[x_1; \sigma_1,\delta_1]\cdots [x_{i-1}; \sigma_{i-1},\delta_{i-1}](-1) \to R[x_1; \sigma_1,\delta_1]\cdots [x_{i-1}; \sigma_{i-1},\delta_{i-1}]
\end{equation}
 is a graded $\sigma_i$-derivation, $2\leq i\leq n$, where $(-1)$ in (\ref{eq.deriteratOre}) denotes a shift.\\
 
Phan \cite{Phan2012} proved that the graded Ore extension $A=R[x; \sigma,\delta]$ is  Koszul if and  only if $R$ is  Koszul. Thus the graded iterated Ore extension $A := R[x_1; \sigma_1,\delta_1]\cdots [x_{n}; \sigma_{n},\delta_{n}]$ is Koszul if and only if $R$ is Koszul. The graded skew Calabi--Yau property of connected graded algebras is preserved by graded Ore extensions (see \cite{Liu2014}). Thus, if $R$ is graded skew Calabi--Yau then the graded iterated Ore extension $A := R[x_1; \sigma_1,\delta_1]\cdots [x_{n}; \sigma_{n},\delta_{n}]$ is graded skew Calabi--Yau.

The class of graded iterated Ore extensions is strictly contained in the class of
graded skew PBW extensions \cite[Remark 2.11]{Suarez2017Calabi}. Let $\mathcal{G}$ be a finite dimensional Lie algebra over $K$ with basis $\{x_1, \dots,
x_n\}$ and $\mathcal{U}(\mathcal{G})$ its enveloping algebra. \emph{The homogenized enveloping
algebra} of $\mathcal{G}$ is $\mathcal{A}(\mathcal{G}):= T(\mathcal{G}\oplus Kz)/\langle
R\rangle$, where $T(\mathcal{G}\oplus Kz)$ is the tensor algebra, $z$ is a new variable,
and $R$ is spanned by $\{z\otimes x-x\otimes z\mid x\in \mathcal{G}\}\cup \{x\otimes y-y\otimes
x-[x,y]\otimes z\mid x,y\in \mathcal{G}\}$ (see \cite[Chapter 12]{Smith1994}). From the PBW theorem for $\mathcal{G}\otimes
K(z)$, considered as a Lie algebra over $K(z)$, we get that
$\mathcal{A}(\mathcal{G})$ is a skew PBW extension of $K[z]$ but it is not an iterated Ore extension. In particular,  $\mathcal{A}(\mathcal{G})$ is a graded skew PBW extension of $K[z]$ but it is not a graded iterated Ore extension.

We say that the graded algebra $A$ is \emph{generated in degree 1} if $A$ is generated as
a $K$-algebra by $A_1$. Note that the algebra $G\mathcal{H}_{p,q}$ is graded with $\deg x=\deg y= 1$ and  $\deg t =2$. Also, $H\mathcal{H}_{p,q}$ is graded with $\deg x=\deg y= \deg t =1$. So $H\mathcal{H}_{p,q}$ is generated in degree 1 but  $G\mathcal{H}_{p,q}$ is not. In contexts such as graded skew PBW extensions and graded iterated Ore extensions the fact that the algebra is generated in degree 1 is of relevance.

\begin{proposition}\label{prop.twoHisgrskpbw}
The algebra $H\mathcal{H}_{p,q}$ is a connected, graded iterated Ore extension and a graded skew PBW extension of $R=K[t]$.
\end{proposition}
\begin{proof}
As the homogeneous component of degree 0 of $H\mathcal{H}_{p,q}$ is $K$, then $\mathcal{H}_{p,q}$ is connected. Now $H\mathcal{H}_{p,q}$ is
the iterated Ore extension $$K[t][x; \sigma_x][y; \sigma_y,
\delta_y],$$ where $\sigma_x(t)=pt$, $\sigma_y(t)=p^{-1}t$,
$\sigma_y(x)=qx$, $\delta_y(t)=0$ and $\delta_y(x)=t^2$. It is clear that $\sigma_x$ and $\sigma_y$ are graded algebra automorphisms and  $\delta_y$ is a graded $\sigma_y$-derivation of degree $1$. So $H\mathcal H_{p,q}$ is indeed a graded iterated Ore extension and therefore a graded skew PBW extension of $K[t]$.
\end{proof}

Note that $G\mathcal{H}_{p,q}$ is an iterated Ore extension and a skew PBW extension. Now, $G\mathcal{H}_{p,q}$ is graded but it is not a graded iterated Ore extension or a graded skew PBW extension, since $\deg t=2$.\\

A graded algebra is right (left) noetherian if and only if it is graded right
(left) noetherian \cite[Proposition 1.4]{Levasseur1992}, which means that every graded right (left) ideal is finitely
generated. In particular, we deduce the follwing.

\begin{corollary}\label{cor.twoqisnoeth}
The algebras $G\mathcal{H}_{p,q}$ and $H\mathcal{H}_{p,q}$ are graded left noetherian and graded right noetherian.
\end{corollary}

\begin{theorem}\label{teo.TwoHKosz}
The  two-parameter quantum Heisenberg algebra $H\mathcal{H}_{p,q}$ is Koszul, graded skew Calabi--Yau of dimension 3 and Artin--Schelter regular.
Furthermore, the Ore extensions $H\mathcal{H}_{p,q}[z; \nu]$ and $H\mathcal{H}_{p,q}[z^{\pm}; \nu]$ are graded Calabi--Yau, where $\nu$ is the Nakayama
automorphism of $H\mathcal{H}_{p,q}$.
\end{theorem}

\begin{proof}
By Proposition \ref{prop.twoHisgrskpbw} we have that $H\mathcal{H}_{p,q}$ is the connected graded iterated Ore extension $K[t][x; \sigma_x][y; \sigma_y,
\delta_y]$. As $K[t]$ is Koszul then $H\mathcal{H}_{p,q}$ is Koszul (see \cite[Corollary 1.3]{Phan2012}). Now as $K[t]$ is graded skew Calabi--Yau of dimension 1 then $K[t][x; \sigma_x]$ is graded skew Calabi--Yau of dimension 2 and thus $K[t][x; \sigma_x][y; \sigma_y,
\delta_y]=H\mathcal{H}_{p,q}$ is graded skew Calabi--Yau of dimension 3, by \cite[Theorem 3.3]{Liu2014}. Since $H\mathcal{H}_{p,q}$ is connected  and graded skew Calabi- Yau then it is Artin-Scelter regular, by \cite[Lemma 1.2]{ReyesRogalskiZhang}. Let $\nu$ be a Nakayama automorphism of  $H\mathcal{H}_{p,q}$ and $z$ be a new variable. As $H\mathcal{H}_{p,q}$ is noetherian, connected, and by the previous considerations it is Koszul and graded skew Calabi--Yau, then $\operatorname{hdet}(\nu) = 1$ \cite[Theorem 6.3]{ReyesRogalskiZhang}, where $\operatorname{hdet}$ is the homological determinant. As, in addition, $H\mathcal{H}_{p,q}$ is Artin--Schelter regular, then by \cite[Proposition 7.3]{ReyesRogalskiZhang} we have that   $H\mathcal{H}_{p,q}[z; \nu]$ and $H\mathcal{H}_{p,q}[z^{\pm}; \nu]$ are graded Calabi--Yau.
\end{proof}
Note that the Calabi--Yau property is not preserved by Ore extensions. For example, the Jordan plane $A = K\langle x, y\rangle/\langle yx-xy-x^2\rangle= K[x][y;\sigma,\delta]$, where $\sigma(x) = x$ and $\delta(x) = x^2$, is a graded Ore extension of a Calabi--Yau algebra $K[x]$,
but $A$ is not Calabi--Yau, since the Nakayama automorphism $\nu$ of the Jordan plane is given by $\nu(x) = x$ and $\nu(y) = 2x + y$, which is not inner. More generally, the algebras of the form $A _h= K\langle x, y\rangle/\langle yx-xy-h(x)\rangle$, where $h(x)$ is an arbitrary polynomial in $x$, is an Ore extension of $K[x]$ which is Calabi--Yau if and only if $\frac{d}{dx}h=0$, which in characteristic $0$ holds only for the commutative polynomial algebra $A_0$ and the Weyl algebra $A_1$, and in prime characteristic $\ell>0$ if $h$ is a polynomial in $x^\ell$ (see~\cite{LS21} for more details).

Gaddis \cite[Proposition 3.2]{Gaddis2016} showed that $G\mathcal{H}_{p,q}$ is Artin--Schelter regular, using the fact that $t$ is a normal regular element.

\begin{corollary}\label{cor.GTwoHKosz}
The Gaddis two-parameter quantum Heisenberg algebra $G\mathcal{H}_{p,q}$ is graded skew Calabi--Yau of dimension 3.
\end{corollary}

\begin{proof}
By \cite[Proposition 3.2]{Gaddis2016} we have that $G\mathcal{H}_{p,q}$ is Artin--Schelter regular of global dimension 3. Then by \cite[Lemma 1.2]{ReyesRogalskiZhang}, $G\mathcal{H}_{p,q}$ is graded skew Calabi--Yau of dimension 3, since it is connected.
\end{proof}

Gaddis \cite[Proposition 3.14]{Gaddis2016} also showed that if $G\mathcal{H}_{p,q}\cong G\mathcal{H}_{p', q'}$ then $(p', q')$ is one of the following tuples: $(p, q)$; $(q, p)$; $(p^{-1}, q^{-1})$;
$(q^{-1}, p^{-1})$, in spite of the seemingly lack of symmetry between $p$ and $q$ in the definition of $G\mathcal{H}_{p,q}$. 

\section{Generalized Heisenberg algebras}\label{sect.GenHeis}
Generalized Heisenberg algebras were formally introduced in \cite{Curado2001}; later, L\"u and Zhao devised in \cite{Lu2015} a more general definition for this type of algebras.
\begin{definition}\label{def.GenHeis}
Given a polynomial $f\in K[t]$, the \emph{generalized Heisenberg}
algebra associated to $f$, denoted $\mathcal{H}(f)$, is the algebra with generators $x$,
$y$ and $t$, with defining relations:
\begin{equation}\label{eq.genHeis}
  tx=xf,\quad yt=fy,\quad yx-xy = f-t.
\end{equation}
\end{definition}
In the present paper we only study generalized Heisenberg algebras $\mathcal{H}(f)$ according to the previous definition, which differs from that in~\cite{Curado2001}.

Many properties of skew PBW extensions have been studied. We are thus interested in knowing when a generalized Heisenberg algebra $\mathcal{H}(f)$ is a skew PBW extension. In the next theorem we give necessary and sufficient conditions for a generalized Heisenberg algebra to be a skew PBW extension.

\begin{theorem}\label{lema.GenHeis} The following are equivalent for a generalized Heisenberg algebra $\mathcal{H}(f)$:
\begin{enumerate}
\item[\rm (i)] $\mathcal{H}(f)$ is a bijective skew PBW extension of endomorphism type of $R=K[t]$;
\item[\rm (ii)] $\mathcal{H}(f)$ is a skew PBW extension of $K$;
\item[\rm (iii)] $\mathcal{H}(f)$  is left or right noetherian;
\item[\rm (iv)] $\deg f=1$
\end{enumerate}
\end{theorem}
\begin{proof}
Suppose that $\mathcal{H}(f)$ is a bijective skew PBW extension of endomorphism type of $R=K[t]$. Since $K[t]$ is noetherian, then so is $\mathcal{H}(f)$. By \cite[Proposition 2.4]{Lopes2017}, $\deg f=1$.

Now suppose that $f(t)=qt+k$ with $k,q\in K$ and $q\neq 0$. Then, by \cite{Lopes2017}, $\mathcal{H}(f)$ is the iterated Ore extension $K[t][x;\sigma_x][y;\sigma_y,\delta_y]$, where $\sigma_x(t)=q^{-1}t-q^{-1}k$,
$\sigma_y(x)=x$, $\sigma_y(t)=qt+k$, all bijective on $R$, and $\delta _y(x)=(q-1)t+k$. It follows that $\mathcal{H}(f)$ is a left free $R$-module, with basis the
set of standard monomials
$\{x^{\alpha_1}y^{\alpha_2}\mid \alpha_1, \alpha_2\in \mathbb{N}\}$.
Then $xt=(q^{-1}t-q^{-1}k)x$,
$yt=(qt+k)y$, and $yx-xy = (q-1)t+k$. Thus the conditions
(\ref{eq.coef}) and (\ref{eq.variab}) of the definition of a skew PBW
extension are satisfied for $R=K[x]$. 

Moreover, from (\ref{eq.genHeis}) we have that $xt-q^{-1}tx=-q^{-1}kx$,
$yt-qty=ky,\quad yx-xy = (q-1)t+k$. Thus the condition
(\ref{eq.variab}) of the definition of a skew PBW extension is satisfied
for $R=K$ and the variables $t,x,y$. If $\deg f>1$ then, for every $\lambda\in K^*$,
$yx-\lambda xy =(1-\lambda)xy + (f-t)\notin K+Kt+Kx + Ky$, which contradicts the relation
(\ref{eq.variab}) of Definition \ref{def.skewpbwextensions}. Finally, if $\deg f=0$, say $f(t)=k\in K$, then for every $\lambda\in K^*$, $xt-\lambda tx=xt-k\lambda x\notin K+Kt+Kx + Ky$, again contradicting relation
(\ref{eq.variab}) of Definition~\ref{def.skewpbwextensions}.

By \cite[Proposition 2.4]{Lopes2017} we have that $\mathcal{H}(f)$
is left (right) noetherian if and only if $\deg f = 1$. Thus,
$\mathcal{H}(f)$ is left (right) noetherian if and only if it is a
skew PBW extension of $K[t]$ or  $K$.
\end{proof}

\begin{proposition}\label{prop.GenHskewy}
Let $0\neq f\in K[t]$. Then a generalized
Heisenberg algebra $\mathcal{H}(f)$ is a skew PBW extension of $R=K\langle
x,t\rangle/\langle tx - xf \rangle$ in the variable $y$.
\end{proposition}
\begin{proof}
Since  $\{x^{\alpha_1}t^{\alpha_2}y^{\alpha_3}\mid \alpha_1,
\alpha_2, \alpha_3\in \mathbb{N}\}$ is a basis of $\mathcal{H}(f)$
then $H(f)$ is a left free $R$-module with basis the set of standard
monomials $\{y^{\alpha} \mid \alpha\in \mathbb{N}\}$. Since
$yt-fy=0\in R$, with $f\neq 0$, and $yx-xy=f-t\in R$, then
condition (\ref{eq.coef}) of Definition \ref{def.skewpbwextensions}
is satisfied. Therefore $\mathcal{H}(f)= \sigma(R)\langle y\rangle$.
\end{proof}

Note that a generalized Heisenberg algebra $\mathcal{H}(f)$ is a graded algebra with $x$, $y$ and $t$ in degree $1$ if and only if $f=t$, in which case $\mathcal{H}(f)= K[t,x,y]$ (see \cite{Lu2015}). Thus, $\mathcal{H}(f)$ is graded (skew) Calabi--Yau or Koszul or Artin--Schelter regular if and only if $f=t$.

The next result motivated the introduction, in the next section, of a broader class which includes both the generalized Heisenberg algebras and the two-parameter quantum Heisenberg algebras, including the enveloping algebra of the Heisenberg Lie algebra itself.

\begin{theorem}\label{prop.qtwoHGenH}
A Gaddis two-parameter quantum Heisenberg algebra $G\mathcal{H}_{p,q}$ is isomorphic to a generalized Heisenberg algebra if and only if either $p=1$ or $q=1$, but not both. In particular, the enveloping algebra $G\mathcal{H}_{1,1}$ of the Heisenberg Lie algebra is not a generalized Heisenberg algebra.
\end{theorem}
\begin{proof}
Suppose that $\mathcal{H}(f)$ is isomorphic to $G\mathcal{H}_{p,q}$, for some $f\in K[t]$ and $p,q\in K^*$. By Corollary~\ref{cor.twoqisnoeth}, $G\mathcal{H}_{p,q}$ is noetherian, so the same holds for $\mathcal{H}(f)$. By~\cite[Proposition 2.1]{Lopes2017}, it follows that $\deg f=1$. So set $f(t)=\lambda t+\mu$, with $\lambda, \mu\in K$ and $\lambda\neq 0$.

\underline{Case 1: $\lambda=1$.} If $\mu=0$ then $f(t)=t$ and $\mathcal{H}(f)$ is the commutative polynomial ring $K[t, x, y]$. Thus, $G\mathcal{H}_{p,q}$ is commutative, whence $0=[y,x]=(q-1)xy+t$. But Proposition~\ref{prop.twoqHisSkew} implies that $\{t^{i}x^{j}y^{k}\mid i,j,k\in \mathbb{N}\}$ is a $K$-basis of $G\mathcal{H}_{p,q}$, which contradicts the previous relation. Therefore, $G\mathcal{H}_{p,q}$ is never commutative and $\mu\neq 0$.

Suppose that $I$ is an ideal of $\mathcal{H}(t+\mu)$ such that $\mathcal{H}(t+\mu)/I$ is commutative. Then $1+I=\mu^{-1}[y,x]+I=\mu^{-1}[y+I,x+I]=0+I$, so $1\in I$. It follows that $\mathcal{H}(t+\mu)$ has no nonzero commutative epimorphic images. On the other hand, there is an algebra homomorphism $\phi:G\mathcal{H}_{p,q}\longrightarrow K[z]$ such that $\phi(x)=\phi(t)=0$ and $\phi(y)=z$, and $\phi$ is surjective. So $G\mathcal{H}_{p,q}/\ker\phi\cong K[z]$ has nonzero commutative epimorphic images, which contradicts the hypothesis that $\mathcal{H}(f)\cong G\mathcal{H}_{p,q}$.

\underline{Case 2: $\lambda\neq1$.} Then there is an isomorphism $\psi:\mathcal{H}(\lambda t)\longrightarrow \mathcal{H}(\lambda t+\mu)$ defined on the generators by $\psi(t)=t+\frac{\mu}{\lambda-1}$, $\psi(x)=x$, $\psi(y)=y$. Indeed, computing in $\mathcal{H}(\lambda t+\mu)$,

\begin{align*}
\psi(tx-\lambda xt)&= \left(t+\frac{\mu}{\lambda-1}\right)x-\lambda x\left(t+\frac{\mu}{\lambda-1}\right)\\
&=tx-\lambda xt + (\lambda-1)^{-1}(\mu-\lambda\mu)x=\mu x-\mu x=0;\\[5pt]
\psi(yx-xy-(\lambda-1)t)&=yx-xy-(\lambda-1)\left(t+\frac{\mu}{\lambda-1}\right)\\
&=(\lambda-1)t+\mu -(\lambda-1)t-\mu=0;
\end{align*}
and similarly $\psi(yt-\lambda ty)=0$. The inverse isomorphism sends $x\in \mathcal{H}(\lambda t+\mu)$ to $x\in \mathcal{H}(\lambda t)$, $y\in \mathcal{H}(\lambda t+\mu)$ to $y\in \mathcal{H}(\lambda t)$ and $t\in \mathcal{H}(\lambda t+\mu)$ to $t-\frac{\mu}{\lambda-1} \in \mathcal{H}(\lambda t)$. Thus, without loss of generality, we can assume that $\mu=0$. So the defining relations of $\mathcal{H}(f)$ are
\begin{equation*}
  tx=\lambda xt,\quad yt=\lambda ty,\quad yx-xy = (\lambda-1) t.
\end{equation*}

Let $I$ be the minimal ideal of $\mathcal{H}(\lambda t)$ such that $\mathcal{H}(\lambda t)/I$ is commutative. Then, since $\lambda\neq 1$, $I$ is generated by the normal element $t$: $I=t\mathcal{H}(\lambda t)=\mathcal{H}(\lambda t)t$.

\underline{Subcase 2A: $p=1$.} Then $t\in G\mathcal{H}_{1,q}$ is central. The minimal ideal $J$ such that $G\mathcal{H}_{1,q}/J$ is commutative is generated by $[y,x]=(q-1)xy+t$.

Suppose, by contradiction, that $q=1$. Then $[y,x]=t$ is central and $J=tG\mathcal{H}_{1,q}$. If $\phi: G\mathcal{H}_{1,q}\longrightarrow \mathcal{H}(\lambda t)$ is an isomorphism, then $\phi(J)=I$. As the only units in $\mathcal{H}(f)$ are the nonzero scalars, this forces $\phi(t)=\gamma t$, for some $\gamma\in K^*$. But this is a contradiction as $t\in G\mathcal{H}_{1,q}$ is central yet $t\in \mathcal{H}(\lambda t)$ is not, because $\lambda\neq 1$. This forces $q\neq 1$, as needed.

\underline{Subcase 2B: $p\neq1$.} Suppose, by contradiction, that $q\neq1$.

The space of maximal ideals of codimension $1$ of $\mathcal{H}(\lambda t)$ (equivalently, as a set, the $1$-dimensional representations of $\mathcal{H}(\lambda t)$) is given by
\begin{equation*}
\{ (x-\alpha, y-\beta, t-\gamma)\mid \alpha, \beta, \gamma\in K,\ \alpha\gamma=\beta\gamma=\gamma=0\}=
\{ (x-\alpha, y-\beta, t)\mid \alpha, \beta\in K\}\cong\mathbb{A}^2,
\end{equation*}
an irreducible algebraic variety of dimension $2$.

On the other hand, assuming that $p,q\neq 1$, the space of maximal ideals of codimension $1$ of $G\mathcal{H}_{p,q}$ is given by
\begin{align*}
\{ (x-\alpha, y-\beta, t-\gamma)\mid \alpha, \beta, \gamma\in K, &\ \alpha\gamma=\beta\gamma=0, \ (1-q)\alpha\beta=\gamma\}=\\
&=\{ (x-\alpha, y-\beta, t)\mid \alpha, \beta\in K,\ \alpha\beta=0\},
\end{align*}
a reducible algebraic variety of dimension $1$.

This contradiction shows that if $p\neq1$ then $q=1$, as stated. This concludes the direct implication.

For the converse implication, and to avoid confusion between the generators of each one of these algebras, we denote the defining generators of $\mathcal{H}(f)$ by $x, y, t$ and the defining generators of $G\mathcal{H}_{p,q}$ by $X, Y, T$.

We assume first that $p=1$ and $q\neq 1$. So we have
\begin{equation*}
YX-qXY=T, \quad XT = TX, \quad YT=TY.
\end{equation*}
We claim that there is an isomorphism $\phi: G\mathcal{H}_{1,q}\longrightarrow \mathcal{H}(q t)$ such that $\phi(X)=\frac{x}{q-1}$, $\phi(Y)=y$ and $\phi(T)=t-xy$. Indeed we have
\begin{align*}
\phi(YX-qXY-T)&=(q-1)^{-1}(y x-qxy)-t+xy\\
&=(q-1)^{-1}((1-q)xy+(q-1)t)-t+xy=0;\\[5pt]
\phi(XT - TX)&=(q-1)^{-1}(x(t-xy)-(t-xy)x)\\&=(q-1)^{-1}x((1-q)t-xy+yx)=0;
\end{align*}
and similarly $\phi(YT - TY)=0$. The inverse map takes $x$ to $(q-1)X$, $y$ to $Y$ and $t$ to $T+(q-1)XY$.

Now, if $p\neq 1$ and $q=1$, we use \cite[Proposition 3.14]{Gaddis2016} and the previous result to obtain
$G\mathcal{H}_{p,1}\cong G\mathcal{H}_{1,p}\cong \mathcal{H}(p t)$.
\end{proof}

\begin{corollary}\label{cor.disyGenHqHandH}
Up to isomorphism, the classes of generalized Heisenberg algebras $\mathcal{H}(f)$ and quantum Heisenberg algebras $\mathcal{H}_{q}$ are disjoint.
\end{corollary}
\begin{proof}
The quantum Heisenberg algebra $\mathcal{H}_{q}$ is the Gaddis two-parameter quantum Heisenberg algebra $G\mathcal{H}_{q,q}$, and the result follows from Theorem~\ref{prop.qtwoHGenH}.
\end{proof}

\section{Quantum generalized Heisenberg algebras}\label{section-QuanVsSkew}

In \cite{Lopes2022Q}, Lopes and Razavinia introduced the quantum
generalized Heisenberg algebras which depend on a parameter $q$ and
two polynomials $f, g\in K[t]$. The class of quantum generalized
Heisenberg algebras includes generalized down-up algebras, the
Heisenberg algebra $\mathcal{H}$, quantum Heisenberg algebras
$\mathcal{H}(q)$,  Gaddis two-parameter quantum Heisenberg algebras
$G\mathcal{H}_{p,q}$, Hayashi two-parameter quantum Heisenberg algebras
$H\mathcal{H}_{p,q}$, the deformations of the enveloping algebra of $\mathfrak{sl}_2$ introduced by Smith in~\cite{spS90}, those introduced by Jing and Zhang in~\cite{JZ95}, and many others.

\begin{definition}[\cite{Lopes2022Q}, Definition 1.1]\label{Def.QGHal} Let $K$ be an arbitrary field and fix $q\in K$ and $f,g\in
K[t]$.
The \emph{quantum generalized Heisenberg algebra}, denoted by $\mathcal{H}_q(f, g)$, is the algebra generated
by $x, y$ and $t$, with defining relations
\begin{equation}\label{eq.QgenHeis}
tx=xf, \quad yt=fy, \quad yx-qxy=g.
\end{equation}
\end{definition}

\begin{remark}\label{rem.relQGHeis}
\begin{enumerate}
\item[\rm (i)] Any generalized Heisenberg algebra $\mathcal{H}(f)$  is  a quantum generalized Heisenberg algebra $\mathcal{H}_q(f, g)$, by setting  $q=1$ and $g= f-t$, i.e.
  $\mathcal{H}(f) = \mathcal{H}_1(f, f - t)$.
\item[\rm (ii)] A Heisenberg algebra $\mathcal{H}$ is  a quantum generalized Heisenberg $\mathcal{H}_q(f, g)$, setting $f=g=t$ and $q=1$.
\item[\rm (iii)] A quantum Heisenberg algebra $\mathcal{H}_q$ is  a quantum generalized Heisenberg $\mathcal{H}_q(f, g)$, setting $f=q^{-1}t$ and $g=t$.
\item[\rm (iv)] A Gaddis two-parameter quantum Heisenberg algebra $G\mathcal{H}_{p,q}$  is a quantum generalized Heisenberg algebra $\mathcal{H}_q(f, g)$, by setting $f=p^{-1}t$, $g=t$ and $q=p$.
\item[\rm (iv)] A Hayashi two-parameter quantum Heisenberg algebra $H\mathcal{H}_{p,q}$  is a quantum generalized Heisenberg algebra $\mathcal{H}_q(f, g)$, by setting $f=p^{-1}t$, $g=t^2$ and $q=p$.
\end{enumerate}
\end{remark}

Cassidy and Shelton in \cite{Cassidy2004} introduced the
\emph{generalized down-up algebra} $L(g, p_1, p_2, p_3)$ as the
unital associative algebra generated by $d, u$ and $t$ with defining
relations
\begin{equation}\label{eq.gendown} dt= p_1td - p_3d,\quad
tu = p_1ut - p_3u, \quad du - p_2ud+g = 0,
\end{equation}
where $p_1, p_2, p_3\in K$ and $g\in K[t]$. The
generalized down-up algebra $L(g, p_1, p_2, p_3)$ is isomorphic to
the quantum generalized Heisenberg algebra
$\mathcal{H}_{p_2}(p_1t-p_3,-g)$ and conversely, any quantum
generalized Heisenberg algebra $\mathcal{H}_q(f, g)$ such that $f =
at+b$, with $a, b\in K$, is a generalized down-up algebra of the
form $L(-g, a, q,-b)$ (see \cite[Proposition 1.3 ]{Lopes2022Q}).

 \begin{proposition}\label{prop.qgenifflin} The quantum generalized Heisenberg
 algebra $\mathcal{H}_q(f, g)$ is a bijective skew PBW extension of $K[t]$  if and only if
 $\deg f=1$ and $q\neq 0$.
 \end{proposition}

 \begin{proof}
The proof is analogous to the proof of Theorem~\ref{lema.GenHeis}, since $g\in K[t]$.
\end{proof}

\begin{proposition}\label{prop.qgenskewK}
A quantum generalized Heisenberg algebra $\mathcal{H}_q(f, g)$ is a
skew PBW extension of $K$ if and only if $q\neq 0$, $\deg f=1$ and
$\deg g\leq 1$.
\end{proposition}
\begin{proof}
Let $f=pt+e$ and $g=ct+d$  with $e,c,d\in K$, $p,q\in K^*$. Let's see that
$\mathcal{H}_q(f, g)$ is a skew PBW extension of $K$
in the variables $x,t,y$. Using \cite[Lemma 2.1]{Lopes2022Q} we have
that $\mathcal{H}_q(f, g)$ is a left free $K$-module, with basis
the set $\{x^it^jy^k \mid i, j, k \in \mathbb{N}\}$. As
$\mathcal{H}_q(f, g)$ is a $K$-algebra, condition
(\ref{eq.coef}) of Definition \ref{def.skewpbwextensions} is immediate.
From (\ref{eq.QgenHeis}) we have that $xt-p^{-1}tx=-p^{-1}ex$,
$yt-pty=ey$ and $yx-qxy = ct+d$, which correspond to the conditions (\ref{eq.variab})
of Definition \ref{def.skewpbwextensions} for  the
variables $x,t,y$. Then $\mathcal{H}_q(f, g)$ is a skew PBW extension of $K$ in the variables $t,x,y$. Note that, in addition,
$\mathcal{H}_q(f, g)$ is bijective.

Suppose now that $\mathcal{H}_q(f, g)$ is a skew PBW extension of $K$. If $q=0$ then from (\ref{eq.QgenHeis}) and (\ref{eq.coef}) we have that  $yx=g=rxy+s$ for some $r\in K^*$ and $s\in K+Kt+Kx + Ky$, which is a contradiction by~\cite[Lemma 2.1]{Lopes2022Q}, so $q\neq 0$.

From (\ref{eq.variab}) there is some $d_{t,y}\in K^*$ such that $yt-d_{t,y}ty\in K+Kt+Kx + Ky$. But by (\ref{eq.QgenHeis}) we have $yt-d_{t,y}ty=(f(t)-d_{t,y}t)y$, which is in $K+Kt+Kx + Ky$ if and only if $f(t)-d_{t,y}t\in K$, again by~\cite[Lemma 2.1]{Lopes2022Q}. Thence, $\deg f=1$.

Similarly, from (\ref{eq.variab}) there is some $d_{x,y}\in K^*$ such that $yx-d_{x,y}xy\in K+Kt+Kx + Ky$. By (\ref{eq.QgenHeis}), $yx-d_{x,y}xy=(q-d_{x,y})xy+g(t)$ and~\cite[Lemma 2.1]{Lopes2022Q} implies that $d_{x,y}=q$ and $\deg g\leq1$.
\end{proof}

The first-named author and Razavinia in \cite[Theorem 4.2]{Lopes2022} gave necessary and sufficient conditions for the isomorphism $\mathcal{H}_q(f, g)\cong \mathcal{H}_q(f', g')$. They also proved  that \[\mathcal{H}_q(f, g)\cong\mathcal{H}_q(f(t-\alpha)+\alpha, g(t-\alpha)),\] for any $\alpha\in K$. Let $f_1=pt$, $f=pt+k$ with $1\neq p\in K^*$ (see in \cite[Lemma 2.1]{Lopes2022}). Then
\begin{align*}
 f_1(t-k(1-p)^{-1})+k(1-p)^{-1}&=p(t-k(1-p)^{-1})+k(1-p)^{-1}\\&=pt-pk(1-p)^{-1}+k(1-p)^{-1}=pt+k=f.
\end{align*}
Then by \cite[Lemma 2.1]{Lopes2022} we have that $\mathcal{H}_q(f_1, g_1)\cong \mathcal{H}_q(f, g)$, i.e., $$\mathcal{H}_q(pt, g(t))\cong \mathcal{H}_q(pt+k, g(t-k(1-p)^{-1})),$$ for $g(t)\in K[t]$ and $1\neq p\in K^*$.\\ 

From Proposition \ref{prop.qgenifflin}  and relations (\ref{eq.QgenHeis}) the
following holds.

\begin{corollary}\label{cor.relGenqHgradPBW} Let  $\mathcal{H}_q(f, g)$ be a quantum generalized Heisenberg algebra with  $f=pt$ and $p, q\in K^*$.
\begin{enumerate}
\item[\rm (i)] If $g=0$ then $\mathcal{H}_q(f, g)=K\langle t,x,y\rangle/\langle xt-p^{-1}tx, yt-pty,yx-qxy\rangle$ is a quantum polynomial algebra and a graded skew PBW extension of $K$ or $K[t]$.
\item[\rm (ii)] If $g$ is a homogeneous polynomial of degree 2, then $\mathcal{H}_q(f, g)=K\langle t,x,y\rangle/\langle xt-p^{-1}tx, yt-pty,yx-qxy-g\rangle$ is a connected graded skew PBW extension of $K[t]$.
\end{enumerate}
 \end{corollary}
\begin{corollary}\label{cor.qgenKoszul} Let $\mathcal{H}_q(f, g)$ be quantum generalized Heisenberg
 algebra with $f=pt$, $p,q\in K^*$ and $g$ a homogeneous polynomial of degree 2. Then $\mathcal{H}_q(f, g)$ is Koszul,  Artin--Schelter regular and graded skew Calabi--Yau of dimension 3.
 \end{corollary}

 \begin{proof}
 $\mathcal{H}_q(f, g)=\mathcal{H}_q(pz, cz^2)$ for $c\in K$, $p,q\in K^*$  is the connected graded iterated Ore extension
\[
K[z][x; \sigma_x][y; \sigma_y, \delta_y],
\]
 where
 \[
\sigma_x(z)=pz,\quad \sigma_y(x) = qx,\quad \sigma_y(z) = p^{-1}z,
\]
 \[
 \delta_y(x) = cz^2,\quad \delta_y(z) = 0.
\]
Then the result follows from Theorem \ref{teo.TwoHKosz}.
 \end{proof}

Note that $\mathcal{H}_q(f, g)$, with $q$, $f$ and $g$ as above, is of type $S'_1$ in the classification of regular algebras of global dimension 3 generated in degree 1, given by Artin and Schelter in \cite[Theorem 10]{ArtinSchelter}.

Let $V$ be a 3-dimensional vector space, $\deg(V) = 1$ and $TV$ the tensor algebra. Fix a basis $\{t, x, y\}$ for $V$. The \emph{cyclic partial derivative}
with respect to $x$ of a word $\Phi$ in the letters $t, x, y$ is
$\partial_x(\Phi):= \sum_{\Phi=uxv}vu$ where the sum is taken over all such factorizations. We extend $\partial_x$ to $TV$ by linearity.
We define $\partial_y$ and  $\partial_t$ in a similar way. The \emph{Jacobian algebra} $J(\Phi)$ associated to $\Phi \in TV$ is the quotient algebra of $TV$ by the ideal generated by the cyclic partial derivatives, i.e., $J(\Phi):=TV/\langle \partial_x(\Phi), \partial_y(\Phi), \partial_t(\Phi)\rangle$. The linear span, $R_{\Phi}:= \operatorname{span}\{\partial_x(\Phi), \partial_y(\Phi),\partial_t(\Phi)\}$, does not depend on the choice of basis for $V$ (see \cite{MoriSmith2017}). Given $\Phi_3\in V^{\otimes 3}$ (in this case $\Phi_3$ is called a \emph{homogeneous potential} of degree 3), Mori and  Smith \cite[Theorem 1.3]{MoriSmith2017} proved that $J(\Phi)$ is graded 3-Calabi--Yau if and only if it is a 3-dimensional Artin--Schelter regular algebra.\\

From this point on we assume that the field $K$ is algebraically closed of characteristic not $2$ or $3$.
\begin{theorem}\label{teo.QgHGradCY}
 Let $\mathcal{H}_q(f, g)$ be quantum generalized Heisenberg algebra with $f=pt$,  $p\in K^*$, $q=p^{-1}$ and let $g$ be a homogeneous polynomial of degree 2. Then $\mathcal{H}_q(f, g)$ is a graded Calabi--Yau algebra  of dimension 3.
\end{theorem}
\begin{proof}
Suppose that $f=pt$,  $p\in K^*$, $q=p^{-1}$ and $g=ct^2$ is a homogeneous polynomial of degree 2. Let $\Phi_3=xyt+pytx-ptxy+yxt-xty-ptyx+3^{-1}pct^3$. Then $\partial_x(\Phi_3)=yt-pty$, $\partial_y(\Phi_3)=tx-pxt$ and $\partial_t(\Phi_3)=xy-pyx+pct^2=-p(yx-p^{-1}xy-ct^2)$. Then the Jacobian algebra  associated to $\Phi_3$ is
$$J(\Phi_3)=K\langle t,x,y\rangle/\langle yt-pty, tx-pxt, yx-qxy-ct^2\rangle= \mathcal{H}_q(f, g).$$
Then the quantum generalized Heisenberg algebra $\mathcal{H}_q(f, g)$ is a Jacobian algebra associated to a homogeneous potential $\Phi_3$. By Corollary \ref{cor.qgenKoszul} $\mathcal{H}_q(f, g)$ is Artin--Schelter regular of dimension 3. Then by \cite[Theorem 1.3]{MoriSmith2017} we have  that $\mathcal{H}_q(f, g)$ is graded Calabi--Yau of dimension 3.
\end{proof}

Berger and Taillefer \cite[Theorem 3.1]{Berger2007} proved that if the \emph{Jacobian algebra} $A=J(\Phi_{N+1})$ is a graded Calabi--Yau algebra of dimension 3, where
$\Phi_{N+1}$ is a homogeneous potential of degree $N +1$, and $\Phi = \Phi_{N+1}+\Phi'= \Phi_{N+1} +\Phi_{N} +\cdots +\Phi_1+\Phi_0$ is a potential with $\deg \Phi_j = j$ for each $0 \leq j\leq N + 1$, then $A' := J(\Phi)$ is a  PBW deformation of $A$. They also proved that if $A=J(\Phi_{N+1})$ is a graded Calabi--Yau algebra of dimension 3 and
$A' := J(\Phi)$ is a PBW deformation of $A$ associated to a potential $\Phi =\Phi_{N+1}+\Phi'$ with
$\deg \Phi' \leq N$, then $A'$ is Calabi--Yau of dimension 3 \cite[Theorem 3.6]{Berger2007}.

\begin{theorem}\label{teo.QgHCY}
Let $\mathcal{H}_q(f', g')$ be quantum generalized Heisenberg algebra for $f',g'\in K[t]$, with $f'=pt+k$, $g'=ct^2+dt+e$ with $c,d,e,k\in K$, $p\in K^*$ and $q=p^{-1}$. Then $\mathcal{H}_q(f', g')$ is a Calabi--Yau algebra  of dimension 3.
\end{theorem}
\begin{proof}
By Theorem \ref{teo.QgHGradCY} and its proof we have that for $f=pt$, $g=ct^2$ and $q=p^{-1}$, the quantum generalized Heisenberg algebra $\mathcal{H}_q(f, g)$ is a graded Calabi--Yau algebra  of dimension 3, which is a Jacobian algebra associated to a homogeneous potential $\Phi_3=xyt+pytx-ptxy+yxt-xty-ptyx+3^{-1}pct^3$. Let $\Phi'=kxy+2^{-1}dt^2+et$  and let $\Phi = \Phi_{3}-\Phi'= xyt+pytx-ptxy+yxt-xty-ptyx+3^{-1}pct^3-kxy-2^{-1}dt^2-et$ be a potentials. Then
$\partial_x(\Phi)=yt-pty-ky=yt-f'y$, $\partial_y(\Phi)=tx-pxt-kx$ and $\partial_t(\Phi)=yx-qxy-ct^2-dt-e$. Then  the Jacobian algebra associated to a potential $\Phi$ is
$$J(\Phi)=K\langle t,x,y\rangle/\langle yt-f'y, tx-xf', yx-qxy-g'\rangle= \mathcal{H}_q(f', g').$$ By \cite[Theorem 3.1]{Berger2007} $\mathcal{H}_q(f', g')$ is a
PBW deformation of $\mathcal{H}_q(f, g)$. Since $\mathcal{H}_q(f, g)$ is a graded Calabi--Yau algebra of dimension 3 and $\deg \Phi'=2\leq \deg (\Phi_3)$ then by \cite[Theorem 3.6]{Berger2007} we have that $\mathcal{H}_q(f', g')$ is Calabi--Yau algebra  of dimension 3.
\end{proof}

\begin{corollary}\label{cor.qHeisCY}
The quantum Heisenberg algebra $\mathcal{H}_q$ is a Calabi--Yau algebra.
\end{corollary}
\begin{proof}
Note that $\mathcal{H}_q= \mathcal{H}_q(f', g')$ for  $f'=q^{-1}t$ and $g'=t$. Then by Theorem \ref{teo.QgHCY} we have that $\mathcal{H}_q$ is Calabi--Yau.
\end{proof}

\begin{corollary}\label{cor.genHeisCY}
If $f=t+k$, with $k\in K$, then the generalized Heisenberg algebra $\mathcal{H}(f)$  is  Calabi--Yau.
\end{corollary}
\begin{proof}
If $f=t+k$ with $k\in K$ then a generalized Heisenberg algebra $\mathcal{H}(f)$ is the quantum generalized Heisenberg algebra $\mathcal{H}_1(f, k)$. Thus by Theorem \ref{teo.QgHCY} we have that $\mathcal{H}(f)$ is Calabi--Yau.
\end{proof}

\section*{Acknowledgement}

The first-named author was partially supported by CMUP, member of LASI, which is financed by national funds through FCT -- Funda\c c\~ao para a Ci\^encia e a Tecnologia, I.P., under the project with reference UID/00144/2025, \url{https://doi.org/10.54499/UID/00144/2025}. He would like to thank the invitation and hospitality of the Universidad Pedag\'ogica y Tecnol\'ogica de Colombia, in Tunja, and his present co-authors for the mathematical discussions held during a visit in 2024. The second-named author was partially supported by Vicerrector\'ia
de Investigaci\'on y Extensi\'on, Universidad Pedag\'ogica y Tecnol\'ogica de Colombia, Tunja, Colombia.


\begin{thebibliography}{20}

\bibitem{ArtinSchelter} Artin, M., Schelter, W. F. (1987). Graded algebras of global dimension 3. \emph{Adv. Math.} 66:171-216.

\bibitem{Berger2007} Berger, R., Taillefer, R. (2007). Poincar\'e-Birkhoff-Witt deformations of Calabi--Yau
algebras. \emph{J. Noncommut. Geom.} 1(2):241-270.

\bibitem{Cassidy2004} Cassidy, T., Shelton, B. (2004). Basic properties of generalized down-up algebras. \emph{J. Algebra} 279(1):402-421.

\bibitem{Curado2001} Curado, E. M. F., Rego-Monteiro,  M. A. (2001). Multi-parametric deformed Heisenberg algebras: a route to complexity. \emph{J. Phys. A: Math. Gen.} 34:3253-3264.

\bibitem{FajardoLibro2020} Fajardo, W., Gallego, C., Lezama, O., Reyes, A., Su\'arez, H., Venegas, H.(2020).
Skew PBW extensions--ring and module-theoretic properties, matrix and Gr\"obner methods, and applications. Springer, Cham.

\bibitem{Gaddis2016} Gaddis, J. (2016). Two-parameter analogs of the Heisenberg enveloping algebra. \emph{Comm. Algebra} 44(11):4637-4653.

\bibitem{GallegoLezama} Gallego, C., Lezama, O. (2011). Gr\"obner bases for ideals of $\sigma-$PBW extensions. \emph{Comm. Algebra} 39(1):50-75.

\bibitem{Ginzburg1}  Ginzburg, V. (2006). Calabi--Yau algebras. arXiv:math.AG/0612139v3.

\bibitem{GomezSuarez2020} G\'omez, J. Y., Su\'arez, H. (2020). Double Ore extensions versus graded skew PBW extensions. \emph{Commun. Algebra} 48(1):185-197.

\bibitem{Hayashi1990} Hayashi, T. (1990). $q$-analogues of Clifford and Weyl algebras-spinor and oscillator representations of quantum enveloping algebras. \emph{Comm. Math. Phys.} 127(1):129-144.

\bibitem{He2010} He, J. W., Van Oystaeyen, F., Zhang, Y. (2010). Cocommutative Calabi--Yau Hopf algebras and deformations. \textit{J.
Algebra} 324(8):1921-1939.

\bibitem{JZ95} Jing, N., Zhang, J. (1995). Quantum {W}eyl algebras and deformations of {$U(g)$}. \emph{Pacific J. Math.} 171(2):437-454.



\bibitem{Kirkman1993} Kirkman, E. E., Small, L. W. (1993). $q$-analogs of harmonic oscillators and related rings. \emph{Israel J. Math.} 81:111-127.

\bibitem{Levasseur1992} Levasseur, T. (1992). Some properties of non-commutative regular graded rings. \emph{Glasglow Math. J.} 34:277-300.

\bibitem{Liu2014} Liu, L. Y.,  Wang, S,  Wu, Q. S. (2014). Twisted Calabi--Yau property of Ore extensions. \emph{J. Noncommut. Geom.} 8(2):587-609.

\bibitem{Lopes2017}  Lopes, S. A. (2017). Non-Noetherian generalized Heisenberg algebras. \emph{J. Algebra Appl.} 16(2):1750064.

\bibitem{Lopes2022Q}  Lopes, S. A., Razavinia, F. (2022). Quantum generalized Heisenberg algebras and their representations. \emph{Comm. Algebra} 50(2):463-483.

\bibitem{Lopes2022}  Lopes, S. A., Razavinia, F. (2022). Structure and isomorphisms of quantum generalized Heisenberg algebras. \emph{J. Algebra Appl.} 21(10):2250204.

\bibitem{LS21}
Lopes, S. A., Solotar, A. (2021). Lie structure on the {H}ochschild cohomology of a family of
  subalgebras of the {W}eyl algebra.  {\em J. Noncommut. Geom.}, 15(4):1373--1407.

\bibitem{Lu2015} L\"u, R., Zhao, K. (2015). Finite-dimensional simple modules over generalized
Heisenberg algebras. \emph{Linear Algebra Appl.} 475:276-291.

\bibitem{MoriSmith2017} Mori, I.,  Smith, S. P. (2017). The classification of 3-Calabi--Yau algebras with 3 generators
and 3 quadratic relations. \emph{Math. Z.} 287(1-2): 215-241.

\bibitem{Phan2012} Phan, C. (2012). The Yoneda algebra of a graded Ore extension. \emph{ Comm. Algebra} 40:834-844.

\bibitem{ReyesRogalskiZhang} Reyes, M.,  Rogalski, D., Zhang, J. J. (2014). Skew Calabi--Yau algebras and homological identities. \emph{Adv.  Math.} 264:308-354.

\bibitem{ReyesSuarez2019-2} Reyes, A., Su\'arez, H. (2020). Skew Poincar\'e-Birkhoff-Witt extensions over weak compatible rings. \emph{J. Algebra Appl.} 19(12):2050225.

\bibitem{ReyesSuarez2019Radicals}Reyes, A., Su\'arez, H. (2021). Radicals and K\"othe's conjecture for skew PBW extensions. \emph{Commun. Math. Stat.} 9(2):119-138.

\bibitem{spS90}
Smith, S. P. (1990). A class of algebras similar to the enveloping algebra of {${\rm sl}(2)$}. {\em Trans. Amer. Math. Soc.}, 322(1):285--314.

\bibitem{Smith1994}   Smith, S. P. (1996). Some finite dimensional algebras related to elliptic curves. Representation
Theory of Algebras and Related topics (Mexico City, 1994), CMS Conf. Proc., Amer.
Math. Soc., Providence, RI, 19:315-348.

\bibitem{Suarez2017Calabi} Su\'arez, H., Lezama, O., Reyes, A. (2017). Calabi--Yau property for graded skew PBW extensions. \emph{Rev. Colombiana Mat.} 51(2):221-239.

\bibitem{Suarez2017} Su\'arez, H. (2017). Koszulity for graded skew PBW extensions. \emph{Comm. Algebra} 45(10):4569-4580.

\bibitem{SuarezChaconReyes2021} Su\'arez, H., Chac\'on, A., Reyes, A. (2022).  On NI and NJ skew PBW extensions. \emph{Comm. Algebra}  50(8):3261-3275.

\bibitem{SuarezReyes2023}  Su\'arez, H., Reyes, A. (2023). $\Sigma$-Semicommutative rings and their skew PBW extensions. \emph{S\~ao Paulo J. Math. Sci.} 17:531-554.

\bibitem{SuarezReyesYesica2023}  Su\'arez, H., Reyes, A. Su\'arez, Y. (2023). Homogenized skew PBW extensions. \emph{Arab. J. Math.} 12:247-263.

\bibitem{SuarezHigueraReyes2024} Su\'arez, H., Higuera, S., Reyes, A. (2024). On $\Sigma$-skew reflexive-nilpotent-property for rings. {\em Algebra Discrete Math.} 37(1):134-159.

\bibitem{Wang} Wang, S.-Q.,  Wu, Q.-S. (2012). A class of AS-regular algebras of dimension five. \emph{J. Algebra} 362:117-144.

\bibitem{Zhang} Zhang, J. J, Zhang, J. (2008). Double Ore extensions. \emph{J. Pure Appl. Algebra} 212:2668-2690.

\bibitem{Zhang2} Zhang, J. J., Zhang, J. (2009). Double extension regular algebras of type (14641). \emph{J.  Algebra} 322(2):373-409.

\bibitem{Zhun} Zhu, C.,  Van Oystaeyen, F., Zhang, Y. (2017). Nakayama automorphisms of double Ore extensions of Koszul regular
algebras. \emph{Manuscripta Math.} 152(3-4):555-584.
\end{thebibliography}
\end{document}